\documentclass[reqno,10pt]{amsart}

\usepackage{amsfonts}
\usepackage{amsmath}
\usepackage{amssymb}
\usepackage{calc}
\usepackage{graphicx}

 \usepackage{verbatim} 
\usepackage{amsthm} 
\usepackage[colorlinks=true, pdfstartview=FitV, linkcolor=black, 
      citecolor=black, urlcolor=black]{hyperref}
\usepackage{multicol}

\numberwithin{equation}{subsection}

\newtheorem{theorem}{Theorem}[subsection]
\newtheorem{lemma}[theorem]{Lemma}
\newtheorem{proposition}[theorem]{Proposition}
\newtheorem{corollary}[theorem]{Corollary}
\newtheorem{conjecture}[theorem]{Conjecture}

\theoremstyle{definition}
\newtheorem{remark}[theorem]{Remark}
\newtheorem{example}[theorem]{Example}

\newcommand{\bfv} {{\bf{v}}}
\newcommand{\bfw} {{\bf{w}}}

\newcommand{\bfu} {{\bf{u}}}

\newcommand{\bfb} {{\bf{b}}}
\newcommand{\bfc} {{\bf{c}}}
\newcommand{\bfm} {{\bf{m}}}
\newcommand{\bbA}{{\mathbb{A}}}

\newcommand{\bbP}{{\mathbb{P}}}
\newcommand{\bF}{{\mathbb{F}}}
\newcommand{\bbK}{{\mathbb{K}}}
\newcommand{\bbR}{{\mathbb{R}}}
\newcommand{\bbQ}{{\mathbb{Q}}}

\newcommand{\bbZ}{{\mathbb{Z}}}
\newcommand{\bbN}{{\mathbb{N}}}
\newcommand{\cD}{{\mathcal{D}}}
\newcommand{\cL}{{\mathcal{L}}}
\newcommand{\cB}{{\mathcal{B}}}
\newcommand{\cQ}{{\mathcal{Q}}}
\newcommand{\cOO}{{\mathcal{O}}}

\newcommand{\cS}{{\mathcal{S}}}
\newcommand{\cK}{{\mathcal{K}}}
\newcommand{\Pic}{{\rm Pic}}
\newcommand{\Mor}{ \overline{\rm NE}}
\newcommand{\NE}{ {\rm NE}}
\newcommand{\Nef}{{\rm Nef}}

\newcommand{\fp} {{\mathfrak{a}}}
\newcommand{\fa} {{\mathfrak{b}}}
\newcommand{\fq} {{\mathfrak{q}}}

\newcommand{\cO}{{\mathcal{O}}}

\begin{document}
\title[Variations on Nagata's Conjecture] {Variations on Nagata's Conjecture}

\author{Ciro Ciliberto}
\address{Dipartimento di Matematica, II Universit\`a di Roma, Italy}
\email{cilibert@axp.mat.uniroma2.it}

\author{Brian Harbourne} 
\address{Department of Mathematics, University of Nebraska, Lincoln, NE 68588-0130 USA}
\email{bharbour@math.unl.edu}

\author{Rick Miranda}
\address{Department of Mathematics, Colorado State University, Fort Collins, CO 80523}
\email{Rick.Miranda@ColoState.Edu}

\author{Joaquim Ro\'e}
\address{Departament de Matem\`atiques, Universitat Aut\`onoma de Barcelona,
Edifici C, Campus de la UAB, 08193 Bellaterra (Cerdanyola del
Vall\`es)} \email{jroe@mat.uab.cat}

\begin{abstract} Here we discuss some variations of Nagata's conjecture  on linear systems of plane curves. The 
most relevant concerns  non-effectivity (hence nefness) of certain rays, which we call \emph{good rays}, in the Mori cone of the blow-up $X_n$ of the plane at $n\ge 10$ general points.  Nagata's original result was the existence of a good ray for $X_n$ with $n\ge 16$ a square number. 
Using degenerations,  we give examples of good rays for $X_n$ for all $n\ge 10$. As with Nagata's 
original result, 
this implies the existence of counterexamples to Hilbert's XIV problem. Finally we show that Nagata's 
conjecture for $n\le 89$ combined with a stronger conjecture for $n=10$ implies
Nagata's conjecture for $n\ge 90$.
\end{abstract}

\date{January 30, 2012}

\maketitle

\tableofcontents

\section*{Introduction}
A fundamental problem in algebraic geometry is understanding which divisor classes on a given variety
have effective representatives. One of the simplest contexts for this problem
is that of curves in the plane, and here already it is 
of substantial interest, and not only in algebraic geometry.
For example, given $n$ sufficiently general points $x_1,\ldots, x_n$ in the complex plane
$\mathbb C^2$, nonnegative integers $m_1,\ldots,m_n$ and an integer $d$,
when is there a polynomial $f\in {\mathbb C}[x,y]$ of degree $d$
vanishing to order at least $m_i$ at each point $x_i$?
Although there is a conjectural answer to this question (the (SHGH) Conjecture; see Conjecture \ref{conj:GHH}
and also \cite{Har04}), the conjecture remains open after more then a half century
of attention by many researchers.

This problem is closely related to the question of what self-intersections occur
for reduced irreducible curves
on the surface $X_n$ obtained by blowing up the projective plane at the $n$ points $x_i$.
Blowing up the points introduces rational curves (infinitely many, in fact, 
when $n>8$) of self-intersection $-1$.
Each curve $C$ on $X_n$ corresponds to a projective plane curve $D_C$ of some degree $d$
vanishing to orders $m_i$ at the points $x_i$; the self-intersection $C^2$
is  $d^2-m_1^2-\cdots-m_n^2$. An example of a curve $D_C$
corresponding to a curve $C$ of self-intersection $-1$ on $X_n$
is the line through two of the points $x_i$, say $x_1$ and $x_2$; in this case,
$d=1$, $m_1=m_2=1$ and $m_i=0$ for $i>2$, so we have  
$d^2-m_1^2-\cdots-m_n^2=-1$.
According to the (SHGH) Conjecture, these $(-1)$-curves should be the 
only reduced irreducible curves of negative self-intersection (see Conjecture \ref{NegativityConj})
but proving that there are no others turns out to be itself very hard
and is still open.

One could hope that a weaker version of this problem might 
satisfy the criterion Hilbert stated in his address to the International Congress in Paris in 1900,
of being difficult enough ``to entice us, 
yet not completely inaccessible so as not to mock our efforts''
(``uns reizt, und dennoch nicht v\"ollig unzug\"anglich, 
damit es unserer Anstrengung nicht spotte'').
In fact, Nagata, in connection with his 
negative solution of the 14-th of the problems Hilbert posed in his address, 
made such a conjecture, Conjecture \ref{conj:nagata}. It is weaker 
than Conjecture \ref{NegativityConj} yet still open for every non-square $n\geq 10$.
Nagata's conjecture
does not rule out the occurrence of 
curves of self-intersection less than $-1$, but it does rule out the worst of them.
In particular, Nagata's conjecture asserts that $d^2\geq nm^2$ must hold
when $n\geq10$, where $m=(m_1+\cdots+m_n)/n$. Thus perhaps there are curves
with $d^2-m_1^2-\cdots-m_n^2<0$, such as the $(-1)$-curves mentioned above,
but $d^2-m_1^2-\cdots-m_n^2$ is (conjecturally) only as negative 
as is allowed by the condition that after averaging the multiplicities $m_i$ 
for $n\geq 10$ one must have $d^2-nm^2\geq 0$. 

What our results here show is that in order to prove Nagata's Conjecture for all $n\geq 10$
it is enough to prove it only for $n<90$, if one can verify a slightly stronger conjecture for $n=10$.
But what we hope is to persuade the reader that it satisfies Hilbert's criteria of being both 
enticing and challenging, and at least not
\emph{completely} inaccessible!

\section{Linear systems on general blow-ups of the plane}

\subsection{Generalities} Fix $n$ points $x_1,\ldots, x_n$ in the complex projective plane $\bbP^2$ (which
will be often assumed to be in \emph{very general} position and called \emph{general}) and 
nonnegative  integers  $d,m_{1},\ldots ,m_{n}$.
We denote by $\mathcal{L}(d;m_{1},\ldots, m_{n})$, 
or simply by $(d;m_{1},\ldots, m_{n})$, 
the linear system of plane curves of \emph{degree} $d$
having \emph{multiplicity} at least $m_{i}$ at the \emph{base point}  $x_i$, for $1\le i\le n$. 
Often we will use exponents to denote repetition of multiplicities.
Sometimes, we may simply denote $\mathcal{L}(d;m_{1},\ldots, m_{n})$
by $\cL$. 

The linear system $(d;m_{1},\ldots, m_{n})$ is the projective space
corresponding to the vector subspace 
$\fp_d\subset H^0(\cO_{\bbP^2}(d))$, 
and $\fp=\oplus_{i=0}^n \fp_d$ is the homogeneous ideal, in 
the coordinate ring $S=\oplus_{i=0}^n
H^0(\cO_{\bbP^2}(d))$ of $\bbP^2$,  of the \emph{fat points scheme}  $\sum_{i=1}^nm_ix_i:= 
\operatorname{Proj}(S/\fp)$.

The \emph{expected dimension} of $(d;m_{1},\ldots, m_{n})$  is
$e({d}; m_{1}\ldots, m_{n})=
max\left\{-1, v({d}; m_{1}\ldots, m_{n})\right\}$
where
$$v({d}; m_{1}\ldots, m_{n})=
 \frac{d(d+3)}{2}-\sum_{i=1}^n\frac{m_i(m_i+1)}{2}$$
is the \emph{virtual dimension} of the system. The system is said to be \emph{special}
if 
$$h({d}; m_{1}\ldots, m_{n})>e({d}; m_{1}\ldots, m_{n})$$
where $h({d}; m_{1}\ldots, m_{n})$ is its \emph{true dimension}. 
In particular, an empty linear system is never special.

We record the following definitions:
\begin{itemize}
\item  $(d;m_{1},\ldots, m_{n})$ is  \emph{asymptotically non-special} (ANS) if there is an integer $y$ such that for all nonnegative integers $x\ge y$ the system $(xd;xm_{1},\ldots, xm_{n})$ is non-special;

\item the \emph{multiplicities vector} $(m_{1}\ldots, m_{n})$ of nonnegative integers is \emph{stably non-special} (SNS) if for all positive integers $d,x$ the linear system $(d;xm_{1},\ldots, xm_{n})$ is non-special.
\end{itemize}

Consider the \emph{Cremona--Kantor (CK) group} $\mathcal G_n$ generated by quadratic transformations based at $n$ general points
$x_1,\ldots, x_n$ of the plane and by permutations of these points (see \cite {dv}).  The group $\mathcal G_n$
 acts on the set of linear systems of the type $(d;m_{1},\ldots, m_{n})$. All systems in the same \emph{(CK)-orbit} (or \emph{(CK)-equivalent}) have the same expected, virtual and true dimension.  
A linear system $(d;m_{1},\ldots, m_{n})$ is \emph{Cremona reduced} if it has minimal degree in its {(CK)-orbit}.  
We note that (CK)-orbits need not contain a Cremona reduced element if  they are orbits of  empty linear systems, but orbits of non-empty linear systems always contain Cremona reduced members.
It is a classical result, which goes back to Max Noether (see, e.g.,  \cite {CaCi}), that a non--empty system $(d;m_{1},\ldots, m_{n})$ with general base points  is Cremona reduced if and only if  the sum of any pair or triple of distinct multiplicities does not exceed $d$. In this case the system is  called \emph{standard}  
and we may assume $m_1\ge \ldots\ge m_n$. 

\subsection{General rational surfaces}
Consider the blow-up $f: X_n\to \bbP^2$ of the plane at $x_1,\ldots, x_n$, which  we call a \emph{general rational surface}. The \emph{Picard group} $\Pic (X_n)$ is the abelian group freely generated by:
\begin {itemize}
\item the \emph{line class}, i.e.,  total transform $L=f^*(\cO_{\bbP^2}(1))$;
\item the classes of the  \emph{exceptional divisors} $E_1,\ldots, E_n$ which are contracted to $x_1,\ldots, x_n$.
\end{itemize}
More generally we may work in the $\bbR$-vector space $N_1(X_n)=\Pic(X_n)\otimes _\bbZ\bbR$. 

We will often abuse notation, identifying divisors on $X_n$ with the corresponding line bundles and  their classes in $\Pic (X_n)$, thus passing from additive to multiplicative notation. We will use the same notation for a planar linear system $\cL=(d;m_{1},\ldots, m_{n})$ and its \emph{proper transform} 
\[\cL=dL-\sum_{i=1}^n m_iE_i\]
on $X_n$. With this convention the integers $d,m_{1},\ldots, m_{n}$ are the components with respect to the ordered  basis $(L, -E_1,\ldots, -E_n)$ of $N_1(X_n)$.
The \emph{canonical divisor} on $X_n$ is $K_n=(-3; -1^n)$ (denoted by $K$ if there is no danger of confusion) and, if $(d;m_{1},\ldots, m_{n})$ is an ample line bundle on $X_n$, then $(d;m_{1},\ldots, m_{n})$ is ANS.

Using the intersection form on $N_1(X_n)$, one can \emph{intersect} and \emph{self-intersect} linear systems $(d;m_{1},\ldots, m_{n})$. Given a linear system $\cL=(d;m_{1},\ldots, m_{n})$, one has
\[v(d;m_{1},\ldots, m_{n})= \frac {\cL^2-\cL \cdot K}2\]
and,  if $d\ge 0$,  Riemann-Roch's theorem says that 
\begin{equation}\label{eq:spec}
\cL \;\text {is special if and only if}\; h^0(\cL)\cdot h^1(\cL)>0.
\end{equation}

\subsection{Special effects} Though \eqref {eq:spec} says that speciality is a cohomological property, the only known reason for speciality comes from geometry in the following way.

Assume we have an \emph{effective} linear system $\cL$, i.e. $h^0(\cL)>0$, and suppose there is an irreducible curve $C$ of arithmetic genus $g$ on $X_n$ such that:
\begin{itemize}
\item $h^2(\cL(-C))=0$, e.g. $h^0(\cL(-C))>0$;
\item $h^1(\cL_{\vert C}) >0$, e.g.  $\cL\cdot C\le g-1$ if $g\ge 2$ and  $\cL\cdot C\le g-2$ if $g\le 1$.
\end{itemize}
Then the \emph{restriction exact sequence}
\[ 0\to \cL(-C)\to \cL \to \cL_{\vert C}\to 0\]
implies that
\[ h^1(\cL)\ge h^1(\cL_{\vert C}) >0\]
hence $\cL$ is special. In this case $C$ is called a \emph{special effect curve} for $\cL$ (see \cite {boc}). For example,  
$C$ is a special effect curve for $\cL$ if $g=0$ and $\cL\cdot C\le -2$, in which case $C$ sits in the base locus of $\cL$. But then $C^2<0$ and therefore $C^2=-1$ if $x_1,\ldots, x_n$ are general points (see  \cite {dF1}). In this case $\cL$ is said to be \emph{$(-1)$-special} and there are plenty of them on $X_n$ as long as $n\geq 1$.

\subsection{The Segre--Harbourne--Gimigliano--Hirschowitz Conjecture}  The only known examples of special linear 
systems on a general rational surface $X_n$ are $(-1)$-special. This motivates the conjecture 
(see  \cite {Seg, Har, Gi, Hirsh, cmsegre}, quoted in chronological order):

\begin{conjecture}[Segre--Harbourne--Gimigliano--Hirschowitz (SHGH)]\label{conj:GHH}A linear system $\cL$ on $X_n$ is special if and only if it is  $(-1)$-special.
\end{conjecture}

It goes back to Castelnuovo
that Conjecture \ref{conj:GHH} holds if $n \leq 9$ (see \cite{castelnuovo}; 
more recent treatments  can be found in \cite{Nag59, Gi, Har, Har85}). 
The general conjecture remains open.

Since standard linear systems are not  $(-1)$-special (see \cite{Har85, Hirsh}), 
an equivalent formulation of the (SHGH) conjecture is: \emph {a standard system of plane curves with general base points is not special}.

Recall that a linear system $\cL$ is  \emph{nef} if $\cL\cdot \cL'\ge 0$ for all effective $\cL'$. 
Since the (CK)-orbit of a nef divisor always contains a Cremona reduced element and hence a standard elementz
(see \cite{Har85}), 
the (SHGH) conjecture implies the following conjecture, which we regard as  a weak form
of (SHGH), a point-of-view justified by Proposition \ref{prop:impl}(ii) below.

\begin{conjecture} \label{conj:LU}  A \emph{nef} linear system $\cL$ on $X_n$ is not special.
\end{conjecture}

This conjecture is also open in general. 

The notion of nefness extends to elements in  $N_1(X_n)$ and 
$\xi\in N_1(X_n)$ is nef if and only if $\lambda\xi$ is nef for all $\lambda>0$. Given a nonzero $\xi\in N_1(X_n)$, the set $[\xi]=\{\lambda\xi: \lambda>0\}$ is called  the \emph{ray} generated by $\xi$. Thus it makes sense to talk of \emph{nef rays}. 

\section{Hilbert's 14-th problem and Nagata's conjecture}

\subsection {Hilbert's 14-th problem} Let $k$ be a field, let $t_1,\ldots, t_n$ be indeterminates over $k$ and let
$\bbK$ be an \emph{intermediate field} between $k$ and $k(t_1,\ldots, t_n)$, i.e.
\[ k\;\subseteq\; \bbK\; \subseteq\; k(t_1,\ldots, t_n).\]

Hilbert's 14-th problem asks: \emph{is $\bbK\cap k[t_1,\ldots, t_n]$ a finitely generated $k$-algebra?}

Hilbert had in mind the following situation coming from invariant theory. Let $G$ be a subgroup of the \emph{affine group}, i.e. the group of automorphisms of  $\bbA_k^n$. Then $G$ acts as a set of automorphisms of the $k$-algebra $k[t_1,\ldots, t_n]$, hence on $k(t_1,\ldots, t_n)$,   and we let $\bbK=k(t_1,\ldots, t_n)^G$ be the field of $G$-invariant elements. Then the question is: \emph {is
\[ k[t_1,\ldots, t_n]^G=\bbK \cap k[t_1,\ldots, t_n]\]
a finitely generated $k$-algebra?}

In \cite {Nag59}, Nagata provided counterexamples to the latter formulation of Hilibert's problem. To do this he used the nefness 
of a certain line bundle of the form $(d;m_1,\ldots, m_n)$ (see \S \ref {ss:nagata} below). 

Hilbert's problem has trivially an affirmative answer in the case $n=1$. The answer is also affirmative for $n=2$, as proved by Zariski in \cite{zar}. Nagata's \emph{minimal counterexample} has $n=32$ and $\dim(G)=13$. Several other counteraxamples have been given by various authors, too long a story to be reported on here. The most recent one is due to Totaro and has $n=3$ (see \cite {tod}).

\subsection{Nagata's Conjecture}  \label{ss:nagata} In his work on Hilbert's 14-th problem, Nagata made the following conjecture:

\begin{conjecture}[Nagata's Conjecture (N)]\label{conj:nagata}  If $n\ge 9$ and $(d;m_1,\ldots,m_n)$ is an effective linear system on $X_n$, then 
\begin{equation}\label{eq:nagata}
\sqrt n\cdot d\ge m_1+\cdots +m_n
\end{equation}
and strict inequality holds if $n\ge 10$. 
\end{conjecture}

Using a degeneration argument, Nagata proved the following result, on 
which his counterexamples to  Hilbert's 14-th problem rely:

\begin{proposition} \label{prop:nagata} (N) holds if $n=k^2$, with $k\ge 3$. 
\end{proposition}

Taking this into account, it is clear that  (N) is equivalent to saying 
that the \emph{Nagata class} $N_n=(\sqrt n,1^n)$, or the \emph{Nagata ray} 
$\nu_n=[\sqrt n,1^n]$ it generates, is nef if $n\ge 9$. Note that 
it suffices to verify (N) for linear systems containing prime divisors.

Let $C$ be an irreducible curve of genus $g$ on $X_n$. If (SHGH) holds, then 
the virtual dimension of $\cOO_{X_n}(C)$ is nonnegative, which reads
\begin{equation}\label{eq:Snag} C^2\ge g-1.\end{equation}
In particular, (SHGH) implies the following conjecture:

\begin{conjecture}\label{NegativityConj}
If $C$ is a prime divisor on $X_n$, then $C^2\geq-1$ with $g=0$ when $C^2=-1$.
\end{conjecture}

\begin{lemma}\label{lem:impl} Conjecture \ref {NegativityConj} implies (N).
\end{lemma}
\begin{proof} 
Suppose $C$ is a prime divisor in $(d;m_1,\ldots,m_n)$  violating (N), i.e., $n\ge9$ and
$\sqrt n\cdot d< m_1+\cdots +m_n$. Then, for $m=(m_1+\cdots +m_n)/n$
and using Cauchy--Schwartz inequality  $m^2\leq (m_1^2+\cdots+m_n^2)/n$, we have
$$d^2<n\frac{(m_1+\cdots +m_n)^2}{n^2}=nm^2\leq m_1^2+\cdots+m_n^2.$$
Thus $C^2<0$ and therefore $C^ 2=-1$ and $C$ has genus $0$  by Conjecture \ref{NegativityConj}. But now we have the contradiction
$1=-C\cdot K_n=3d-(m_1+\cdots+m_n)\leq \sqrt{n}d-(m_1+\cdots+m_n)<0$.\end{proof}

Hence Conjecture \ref{NegativityConj} can be regarded as a strong form of (N). 
The aforementioned result in \cite {dF1} (that $C^2\geq -1$ if $C\subset X_n$ is irreducible and rational)
yields \eqref {eq:Snag} if $g=0$. If $g=1$ and $C^2=0$, then 
$C$ is (CK)-equivalent to $(3;1^9, 0^{n-9})$ (see \cite {cmsegre}). 
Thus the following conjecture (see  \cite[Conjecture 3.6]{GHI}) is at least plausible.

\begin{conjecture} [Strong Nagata's Conjecture (SN)] If $C$ is an irreducible curve of genus $g>0$ on $X_n$, then $C^2>0$ unless $n\ge 9$, $g=1$ and $C$ 
is (CK)--equivalent to $(3;1^9, 0^{n-9})$, in which case $C^2=0$. 
\end{conjecture} 

\begin{proposition}\label{prop:impl} We have the following:
\begin{itemize}
\item [(i)] (SHGH) implies  (SN) (in particular (SN) holds for $n\le 9$);
\item[(ii)] (SHGH) holds if and only if (SN) and Conjecture \ref{conj:LU} both hold;
\item [(iii)]  (SN) implies (N).
\end{itemize}
\end{proposition}
\begin{proof} Part (i),  hence also the forward implication of (ii), is clear. 
As for the reverse implication in (ii), note thet (SN) implies Conjecture \ref{NegativityConj}, and Conjecture \ref{NegativityConj} 
together with Conjecture \ref{conj:LU} is the formulation of (SHGH) given in \cite{Har}. 
Finally we prove part (iii), hence we assume $n\ge 9$. 
Let $C$  be an irreducible curve in $(d;m_1,\ldots,m_n)$ on $X_n$. If 
 $C^2\ge 0$ then $d^2\ge m_1^2+\cdots+m_n^2$. By the Cauchy--Schwartz inequality this implies
\eqref {eq:nagata}.
Equality  holds if and  only if $m_1=\cdots=m_n=m$ and $d=m\sqrt n$, hence $n$ is a square, which is only possible if $n=9$ by  Proposition \ref {prop:nagata}. 
If $C^2<0$, then $g=0$ and $C^2=-1$, so that
\[ C\cdot N_n=C\cdot (N_n-K_n)- C\cdot K_n= (\sqrt n-3) C \cdot L+1\ge 1.\] 
\end{proof}

\section{The Mori cone viewpoint}\label{sec:Mori}

\subsection{Generalities}  A class $\xi\in N_1(X_n)$ is \emph{integral} [resp. \emph{rational}]  if it sits in $\Pic(X_n)$ [resp. in $\Pic(X_n)\otimes _\bbZ \bbQ$]. A ray in $N_1(X_n)$ is \emph{rational} if it is generated by a rational class.  A rational ray in $N_1(X_n)$ is \emph{effective} if it is generated by an effective  class. 


The \emph{Mori cone} $\Mor(X_n)$ is the closure in $N_1(X_n)$  of the set $\NE(X_n)$ of all effective rays, and it is the dual of the \emph{nef cone} $\Nef(X_n)$ which is the closed cone
described by all nef rays. 

A \emph{$(-1)$-ray} in $N_1(X_n)$ is a ray generated by a $(-1)$-\emph{curve}, i.e., a smooth, irreducible, rational curve $C$ with $C^2=-1$ (hence $C\cdot K_n=-1$). 

\emph{Mori's Cone Theorem} says that
\[\Mor(X_n)= \Mor(X_n)^\succcurlyeq + R_n \]
where $\Mor(X_n)^\succcurlyeq$ [resp. $\Mor(X_n)^\preccurlyeq$]  is the subset of $\Mor(X_n)$ described by rays generated by nonzero classes $\xi$ such that $\xi\cdot K_n\ge 0$ [resp. $\xi\cdot K_n\le 0$] and
\[R_n\;= \sum_{\rho\;\text {a}\; (-1)-\text{ray}} \rho\; \subseteq\;  \Mor(X_n)^\preccurlyeq.\]
We will denote by $\Mor(X_n)^\succ$ [resp. $\Mor(X_n)^\prec$]  the interior of $\Mor(X_n)^\succcurlyeq$ [resp. $\Mor(X_n)^\preccurlyeq$].

Concerning (SHGH), the situation is well understood for classes in $R_n$, in view of this result (see \cite {LafUg}):

\begin{theorem}\label{thm:LU}  An effective and nef linear system on $X_n$ with class in $R_n$  is non-special. 
\end{theorem}

The \emph{nonnegative cone} $\cQ_n$ in $N_1(X_n)$ is the cone of classes $\xi$ such that $\xi\cdot L\ge 0$ and $\xi^2\ge 0$, whose boundary, which is a quadric cone, we denote by $\partial\cQ_n$. By Riemann-Roch's theorem one has $\cQ_n\subseteq \Mor(X_n)$.  We will use the obvious notation
$\cQ_n^\succcurlyeq, \cQ_n^\preccurlyeq, \cQ_n^\succ, \cQ_n^\prec$ to denote the intersection of $\cQ_n$ with 
$\Mor(X_n)^\succcurlyeq$ etc., and similarly for $\partial\cQ_n$.

The situation is quite different according to the values of $n$:
\begin{itemize}
\item [(i)] The \emph{Del Pezzo case} $n\le 8$. Here $-K_n$ is ample, hence $\Mor(X_n)=R_n$. There are only finitely many $(-1)$-curves on $X_n$, hence $\Mor(X_n)$ is polyhedral and $\Mor(X_n)\subseteq \Mor(X_n)^\prec$. If $\kappa_n=[3,1^n]$ is the \emph {anticanonical ray}, then $\kappa_n$ is in the interior of $\cQ_n$.

\item [(ii)] The \emph{quasi Del Pezzo case} $n=9$.  Here $-K_9$ is an irreducible curve with self-intersection 0. Hence $\kappa_9$ is nef, sits on  $\partial\cQ_9$,  and the tangent hyperplane to $\partial\cQ_9$ at $\kappa_9$ is the hyperplane $\kappa_9^\perp$  of classes $\xi$ such that $\xi\cdot K_9=0$.
Then $\Mor(X_9)^\succcurlyeq=\kappa_9$ and $\Mor(X_n)=\kappa_9+R_9\subseteq   \Mor(X_n)^\preccurlyeq$. There are infinitely many $(-1)$-curves on $X_9$,  and $\kappa_9$ is the only limit ray of  $(-1)$-rays. The anticanonical ray $\kappa_9$ coincides with the Nagata ray $\nu_9$.

\item [(iii)] The \emph{general case} $n\ge 10$.  Here $-K_n$ is not effective, and has negative self-intersection $9-n$. Hence $\kappa_n$ lies off $\cQ_n$, which in turn has non-empty  intersection with both $\Mor(X_n)^\succ$ and $\Mor(X_n)^\prec$. There are infinitely many $(-1)$-curves on $X_n$,  whose rays lie in $\Mor(X_n)^\prec$ and their limit rays 
lie at the intersection of $\partial\cQ_n$ with the hyperplane $\kappa_n^\perp$. The Nagata ray $\nu_n$ sits on $\partial\cQ_n^\succ$. The plane  joining the rays $\kappa_n$ and $\nu_n$ is the \emph{homogeneous slice}, formed by the classes of \emph{homogeneous linear systems} of the form $(d;m^n)$, with $d\ge 0$. 
\end{itemize}

For information on the homogeneous slice, and relations between (N) and (SHGH) there, see \cite {CM11}.

\subsection {More conjectures}\label{ssec:more}  The following conjecture is in \cite {dF10}. Taking into account  the aforementioned result in \cite {dF1}, it would be a consequence of  (SN).

\begin {conjecture}
\label{conj:df} If $n\ge 10$, then 
\begin{equation}\label{eq:df}
\Mor(X_n)=\cQ_n^\succcurlyeq \;+\; R.
\end{equation}
\end{conjecture}

Let $D_n=(\sqrt {n-1}, 1^n)\in N_1(X_n)$ be the \emph{de Fernex} point  and $\delta_n=[\sqrt {n-1}, 1^n]$ the corresponding ray. One has $D_n^2=-1$, $D_n\cdot K_n=n-3\sqrt{n-1}=\frac{n^2-9n+9}{n+3\sqrt{n-1}}>0$ 
for $n\geq 8$ and, if $n=10$, $D_n=-K_n$.  We will denote by  $\Delta_n^\succcurlyeq$ [resp. $\Delta_n^\preccurlyeq$] the set of classes $\xi\in N_1(X_n)$ such that $\xi\cdot D_n\ge 0$ [resp. $\xi\cdot D_n\le 0$]. 

One has (see \cite {dF10}):

\begin{theorem}\label{thm:df} If $n\ge 10$ one has:
\begin{itemize}
\item [(i)] all $(-1)$-rays lie in the cone $\cD_n:=\cQ_n-\delta_n$;
\item [(ii)] if $n=10$, all $(-1)$-rays lie  on the boundary  of the cone $\cD_n$;
\item [(iii)] if $n>10$, all $(-1)$-rays lie  in the complement of the cone $\cK_n:=\cQ_n-\kappa_n$;
\item [(iv)] $\Mor(X_n)\subseteq \cK_n+R$;
\item [(v)] if  Conjecture \ref {conj:df} holds, then 
\begin{equation}\label{eqdf2}\Mor(X_n) \cap \Delta_n^\preccurlyeq=\cQ_n\cap \Delta_n^\preccurlyeq.\end{equation}
\end{itemize}
\end{theorem}

\begin{remark} As noted in \cite {dF10}, Conjecture \ref {conj:df} does not imply  that $\Mor(X_n)^\succcurlyeq=\cQ_n^\succcurlyeq$, unless $n=10$, in which case this is exactly what it says (see Theorem \ref {thm:df}(v)). Conjecture \ref {conj:df} implies (N) but not (SN).
\end{remark}

Consider the following:

\begin{conjecture}[The $\Delta$-conjecture ($\Delta$C)] \label{conj:CHMR} If $n\ge 10$ one has 
\begin{equation}\label{eq:dCHMR}
\partial\cQ_n\cap \Delta_n^\preccurlyeq\subset \Nef(X_n).\end{equation}
\end{conjecture}

\begin{proposition}\label{prop:iml}  If  ($\Delta$C) holds, then
\begin{equation}\label{eq:dCHMR2}
\Mor(X_n) \cap \Delta_n^\preccurlyeq=\Nef(X_n) \cap \Delta_n^\preccurlyeq=\cQ_n\cap \Delta_n^\preccurlyeq.
\end{equation}
\end{proposition}

\begin{proof}   By \eqref {eq:dCHMR} and  by 
the convexity of $\Nef(X_n)$ one has  $\cQ_n\cap \Delta_n^\preccurlyeq\subseteq \Nef(X_n) \cap \Delta_n^\preccurlyeq$. Moreover  $\Nef(X_n) \cap \Delta_n^\preccurlyeq\subseteq \Mor(X_n) \cap \Delta_n^\preccurlyeq$. Finally \eqref {eq:dCHMR} implies  \eqref {eqdf2}   because $\Mor(X_n)$ is dual to $\Nef(X_n)$. 
\end{proof}

The following proposition indicates that Nagata-type conjectures we are discussing here can be interpreted as asymptotic forms of the (SHGH) conjecture. 

\begin{proposition}\label{prop:ANS} Let $n\ge 10$. 
\begin{itemize}
\item [(i)] If ($\Delta$C) holds, then all  classes  in $\cQ_n\cap \Delta_n^\preccurlyeq- \partial\cQ_n\cap \Delta_n^\preccurlyeq$ are ample and therefore, if integral, they are (ASN);
\item [(ii)] If (SN) holds, then a rational class  in $\cQ_n^\succcurlyeq- \partial\cQ_n^\succcurlyeq$ is (ASN) unless it has negative intersection with some $(-1)$-curve.
\end{itemize} 
\end{proposition} 

\begin{proof} Part (i) follows from Proposition \ref {prop:iml} and  the fact that the ample cone is the interior of the nef cone (by Kleiman's theorem, see \cite {Kl}). 

As for part (ii), if $\xi\in \cQ_n^\succcurlyeq- \partial\cQ_n^\succcurlyeq$ is nef, then it is also big. 
If $C$ is an irreducible curve such that $\xi\cdot C=0$, then $C^2<0$ by the index theorem, hence $C$ is a $(-1)$-curve. Contract it, go to $X_{n-1}$ and take the class $\xi_1\in N_1(X_{n-1})$ which pulls back to $\xi$. Repeat the argument on $\xi_1$, and go on. At the end we find a class $\xi_i\in N_1(X_{n-i})$ for some $i\le n$, which is ample by Nakai-Moishezon criterion, and the (ASN) follows for $\xi$.

If $\xi$ is not nef and $C$ is an irreducible curve such that $\xi\cdot C<0$, then $C^2<0$ hence $C$ is a $(-1)$-curve.  \end{proof} 

One can give a stronger form of   $(\Delta C)$. 

\begin{lemma}\label{lem:nef} Any rational, non-effective ray in $\partial\cQ_n$ is nef and it is extremal for both $\Mor(X_n)$ and $\Nef(X_n)$. Moreover it lies in $\partial\cQ_n^\succcurlyeq$. 
\end{lemma}

\begin{proof} Let $\xi$ be a generator of the ray and let $\xi=P+N$ be the Zariski decomposition of $\xi$. Since the ray is not effective, one has $P^2=0$. Since $\xi^2=0$, then $N^2=0$, hence $N=0$, proving that $\xi$ is nef. 

Suppose that $\xi=\alpha+\beta$, with $\alpha,\beta\in \Mor(X_n)$. Then $\xi^2=0$,  $\xi\cdot \alpha\ge 0$ and $\xi\cdot \beta\ge 0$, imply $\alpha^2=-\alpha\cdot \beta=\beta^2$ which yields that $\alpha$ and $\beta$ are proportional. This shows that the ray is extremal for $\Mor(X_n)$. The same proof shows that  it is extremal also for $\Nef(X_n)$.   

The final assertion  follows by the Mori's Cone theorem. \end{proof}

A rational, non-effective ray in $\partial\cQ_n$  will be called a \emph{good ray}. An  irrational, nef ray in $\partial\cQ_n$  will be called a \emph{wonderful ray}. No wonderful ray has been detected so far. The following  is clear:

\begin{lemma} \label{lem:wonder} Suppose that $(\delta;m_1,\ldots,m_n)$ generates either a good or wonderful ray. If $(d;m_1,\ldots, m_n)$ is an
effective linear system  then
\[d>\delta =\sqrt {\sum_{i=1}^n m_i^2}.\]
\end{lemma}

The following conjecture  implies  ($\Delta$C).

\begin{conjecture}[The strong $\Delta$-conjecture (S$\Delta$C)] \label{conj:SCHMR} If $n>10$, all rational rays in $\partial\cQ_n\cap \Delta_n^\preccurlyeq$  are non-effective.  If $n=10$, a rational ray in  $\cQ_{10}\cap \Delta_{10}^\preccurlyeq=\cQ_{10}^\succcurlyeq$ is non--effective, unless it is generated by a curve 
(CK)--equivalent to  $(3;1^9, 0)$.
\end{conjecture} 

\begin{proposition} \label{prop:SN10} For $n=10$,  (S$\Delta$C) is equivalent to (SN).
\end{proposition}

\begin{proof} 
If  (S$\Delta$C) holds then clearly (SN) holds. Conversely, assume (SN) holds, consider a
rational effective ray in   $\partial\cQ_{10}^\succcurlyeq$ and let $C$ be an effective divisor
in the ray. Then $C=n_1C_1+\cdots +n_hC_h$, with $C_1,\ldots, C_h$ distinct irreducible curves 
and $n_1,\ldots, n_h$ positive integers. One has  $C_i\cdot C_j\ge 0$, hence
$C_i\cdot C_j=0$ for all $1\le i\le j\le h$. This  clearly implies $h=1$, hence the assertion.  \end{proof}

By the proof of Proposition \ref {conj:CHMR}, any good ray  gives 
a constraint on $\Mor(X_n)$, so it is useful to find  good rays.  Even better would be to find 
wonderful rays.  We will soon give more reasons for searching for such rays (see \S \ref {sec:goodrays}). 

\begin{example}\label{ex:seq}  Consider the family of linear systems
\[ \cB=\{ B_{q,p}:= (9q^2+p^2; 9q^2-p^2, (2qp)^9):   (q,p)\in \bbN^2, q\le p \}\]
generating rays in $\partial\cQ_{10}^\succcurlyeq$. Take a sequence $\{(q_n,p_n)\}_{n\in \bbN}$ such that $\lim_n \frac {p_n+q_n}{p_n}=\sqrt {10}$. For instance take $\frac {p_n+q_n}{p_n}$ to be the convergents of the periodic continued fraction expansion of $\sqrt {10}=[3; \overline{6}]$, so that 
\[ p_1=2, \; p_2=13, \; p_3=80,\ldots \;\; q_1=1, \; q_2=6, \; q_3=37, \ldots. \]
The sequence of rays $\{[B_{q_n,p_n}]\}_{n\in \bbN}$ converges to the Nagata ray $\nu_{10}$. If we knew that the rays of this sequence are good, this would imply (N) for $n=10$. 
\end{example}

A way of searching for good rays is the following (see \S \ref {sec:goodraysplus}). Let $(m_{1}\ldots, m_{n})$ be a (SNS) multiplicity  vector, with $d=\sqrt {\sum_{i=1}^n m_i^2}$ an integer such that $3d<\sum_{i=1}^n m_i$. Then $[d;m_{1}\ldots, m_{n}]$ is a good ray. We will apply this idea in \S \ref {ssec:base}. 

\renewcommand{\thesection}{\arabic{section}}
\renewcommand{\theequation}{\thesection.\arabic{equation}}
\setcounter{equation}{0}

\renewcommand{\thetheorem}{\thesection.\arabic{theorem}}
\setcounter{theorem}{0}

\section{Good rays and counterexamples to Hilbert's 14-th problem}\label{sec:goodrays}

In this section we show that  any good or wonderful  ray for $n\ge 10$ provides a counterexample to Hilbert's 14-th problem.
The proof follows  Nagata's original argument in \cite {Nag59}, which we briefly recall.

Let $\bF$ be a field. Let ${\bf X}=(x_{ij})_{1\le i\le 3;1\le j\le n}$ be a matrix of indeterminates over $\bF$ and consider the field
$k=\bF [{\bf X}]:= \bF[x_{ij}]_{1\le i\le 3;1\le j\le n}$ (we use similar vector notation later). The points $x_j=[x_{1j}, x_{2j}, x_{3j}]\in \bbP^2_{k}$, $1\le j\le n$, may be 
seen as $n$ general points of $ \bbP^2_{\bF}$.   The subspace $V\subset k^n$ formed by all vectors $\bfb=(b_1,\ldots,b_n)$ such that
${\bf X}\cdot \bfb^t={\bf 0}$,   is said to be \emph{associated} to $x_1,\ldots, x_n$. 

Fix  a multiplicities vector $\bfm=(m_1,\ldots, m_n)$ of positive integers and consider the subgroup $H$ of the multiplicative group $(k^*)^n$ formed by all  vectors $\bfc=(c_1,\ldots,c_n)$ such that $\bfc ^{\bf m}:=  c_1^{m_1} \cdots c_n^{m_n}=1$. We set $\delta=\sqrt {\sum_{i=1}^n m_i^2}$.

Let $\bfu=(u_1,\ldots, u_n)$ and $\bfv=(v_1,\ldots, v_n)$ be vectors of indeterminates over $k$, and consider $k[\bfu,\bfv]$.  The group $G=H\times V$ acts on the $k$-algebra $k[\bfu,\bfv]$ in the following way: if $\sigma=(\bfc,\bfb)$ and $c=c_1\cdots c_n$, then
\[ \sigma (u_i)=\frac {c_i}c(u_i+b_iv_i), \;  \sigma (v_i)=c_iv_i\; {\text {for}}\; 1\le i\le n.\]

\begin{theorem}\label{thm:bignag1} If $(\delta;m_1,\ldots, m_n)$ generates a good or a wonderful ray, then the $k$-algebra $A=k[\bfu,\bfv]^G$ is not finitely generated. 
\end{theorem}

\begin{proof}
The elements  $t:= \bfv^\bfm$ and 
\[ w_i= \sum_{j=1}^n x_{ij} (v_1\cdots v_{j-1} u_jv_{j+1}\cdots v_n),\; \text{for}\; 1\le i\le 3\]
are in $A$. Set $\bfw=(w_1,w_2,w_3)$, which is a vector of indeterminates on $k$. Then $S:= k[\bfw]$ is the homogeneous coordinate ring of 
$\bbP^2_{k}$.
 Imitating the argument in \cite [Lemma 2] {Nag59}, one proves that 
$A= k[\bfu, \bfv]\cap k(\bfw,t)$ and, as a consequence (see \cite [Lemma 3] {Nag59}), that $A$ consists of all sums $\sum_{i\in \bbZ} a_i t^{-i}$, such that  $a_i\neq 0$ for finitely many $i\in \bbZ$, $a_i\in S$  for  all $i\in \bbZ$, and
$a_i\in \fa_i:= \bigcap_{j=1}^n \mathfrak{p}_j^{im_j}$ where
$\mathfrak{p}_j$ is the homogeneous ideal of the point $x_j$.

By \cite [Lemma 3] {Nag59}, to prove that $A$ is not finitely generated it suffices to show that 
\begin{equation}\label{eq;clain}
\text {for all positive}\;  m\in \bbZ,\;  \text{there is a positive}\; \ell\in \bbZ\; \text{such that}\;  \fa_m^\ell\neq \fa_{m\ell}.
\end{equation} 
This is proved as in  \cite [Lemma 1] {Nag59}. Indeed, let $\alpha(\fq)$ be the minimum degree of a  polynomial in a homogeneous ideal $\fq$ of $S$. Since 
\[ v(d; mm_1,\ldots, mn_n)=\frac {d^2-m^2 \delta^2} 2 +\ldots, \]
where $\ldots$ denote lower degree terms, we have $\lim_{m\to \infty} \frac {\alpha(\fa_m)}m \le \delta$. Since 
$(\delta;m_1,\ldots, m_n)$ is nef, we have $\frac {\alpha(\fa_m)}m \ge \delta$ for all positive integers $m$. Hence $\lim_{m\to \infty} \frac {\alpha(\fa_m)}m = \delta$. By Lemma \ref {lem:wonder}, one has
\[\frac {\alpha(\fa_m^\ell)}{m\ell}= \frac {\alpha(\fa_m)}{m} >\delta\]
from which \eqref {eq;clain} follows. \end{proof}

\renewcommand{\theequation}{\thesection.\arabic{equation}}
\setcounter{equation}{0}

\renewcommand{\thetheorem}{\thesubsection.\arabic{theorem}}
\setcounter{theorem}{4}

\section{Existence of good rays}\label{sec:goodraysplus}

\subsection{The existence theorem} 

\begin{theorem}[Existence Theorem (ET)]\label{thm:main}
For every $n\ge 10$, there are good rays in $\Mor(X_n)$.
\end{theorem}

The proof  goes by induction on $n$ (see \S \ref {ss:proof}).
The \emph{induction step} is based on the following proposition:

\begin{proposition}   \label{prop:main} Set $n=s+t-1$, with $s,t$ positive integers. Assume $D=(\delta;\mu_1,\ldots, \mu_s)\in \Nef(X_s)$ with $\mu_1,\ldots, \mu_s$ nonnegative, rational numbers. Let $\nu_1,\ldots, \nu_s$ 
be nonnegative rational numbers
such that $m_1=\sum_{i=1}^s \mu_i\nu_i$ is an integer, and
let  $d, m_2,\ldots, m_t$ be nonnegative rational numbers such that 
$C=(d;m_1,\ldots,m_t)$ generates a non-effective ray in $N_1(X_t)$. Then for every rational number $\eta\ge \delta$, $C_\eta=(d; \eta\nu_1,\ldots, \eta\nu_s, m_2,\ldots, m_t)$ generates a non-effective ray in $N_1(X_n)$. 
\end{proposition} 

The proof of Proposition \ref {prop:main} relies on a degeneration argument introduced in \cite {CM98a}
(see also \cite {CDMR11, CM11}), which is reviewed in \S \ref {ssec:1stdeg}.
The next corollary  shows how  Proposition \ref {prop:main} may be applied  to inductively prove Theorem \ref {thm:main}. 

\begin{corollary}\label{cor:main} Same setting as in Proposition \ref {prop:main}. Assume that:
\begin{itemize}
\item [(i)] $C$ generates a good ray in $N_1(X_t)$;
\item [(ii)] $D\in \Nef(X_s)$ and $D^2=0$;
\item [(iii)] ${\sum_{i=1}^s\nu_i^2}\ge 1$ (this happens if  $\nu_1,\ldots, \nu_s$ are integers) and
$(\nu_1,\ldots, \nu_s)$ is proportional to $(\mu_1,\ldots, \mu_s)$.
\end{itemize} 
Then  $C_\delta\in \partial\cQ_n^\succcurlyeq$ is nef. If   ${\sum_{i=1}^s\nu_i^2}=1$ then $C_\delta\cdot K_n=C\cdot K_t$. 
\end{corollary}
\begin{proof} For $\eta \geq \delta$, one has
 \[
 \begin{aligned}
C_\eta\cdot  K_n= C\cdot K_t-m_1+\eta \sum_{i=1}^s \nu_i^2=
C\cdot K_t+ \eta  \sum_{i=1}^s \nu_i^2 - \sqrt {\sum_{i=1}^s\mu_i^2} \sqrt {\sum_{i=1}^s\nu_i^2}\ge\cr 
=C\cdot K_t+ \eta  \sum_{i=1}^s \nu_i^2 - \delta \sqrt {\sum_{i=1}^s\nu_i^2}
\ge  C\cdot K_t+ \delta \sqrt {\sum_{i=1}^s\nu_i^2} \big(\sqrt {\sum_{i=1}^s\nu_i^2}-1  \big)\ge C\cdot K_t.
 \end{aligned}
 \]
Since  $C\cdot K_t\ge 0$, then also $C_\eta\cdot K_n\ge 0$, hence $C_\delta\cdot K_n\ge 0$,  and   ${\sum_{i=1}^s\nu_i^2}=1$ yields $C_\delta\cdot K_n=C\cdot K_t$. Moreover
\[C_\eta^2= C^2+ m_1^2-\eta^2 \sum_{i=1}^s\nu_i^2= 
 C^2+  \big ( (\sum_{i=1}^s\mu_i^2)-\eta^2\big ) \sum_{i=1}^s\nu_i^2= C^2+  \big (\delta^2-\eta^2) \sum_{i=1}^s\nu_i^2\le C^2=0,\]
in particular $C^2_\delta=0$.      
  
 Assume $C_\delta$ is not nef, hence there is an irreducible curve $E$ such that $C_\delta\cdot E<0$ and $E^2<0$. Take $\eta\ge \delta$  close to $\delta$ and rational. 
 Set $E_\epsilon=\epsilon E+C_\eta$,  with $\epsilon \in \bbR$. One has $E_\epsilon^2= \epsilon^2E^2+2\epsilon(C_\eta\cdot E)+ C_\eta^2$ and  $(C_\eta\cdot E)^2 -C_\eta^2\cdot E^2>0$ because it is close to $(C_\delta\cdot E)^2>0$. Then
 \[ \tau= \frac {-(C_\eta\cdot E) - \sqrt {(C_\eta\cdot E)^2 -C_\eta^2\cdot E^2}} {E^2}\]
 is negative, close to $0$ and such that $E_\tau^2=0$, and $E_\epsilon^2>0$ for $\epsilon<\tau$ and close to $\tau$. Then for these values of $\epsilon$ the class $C_\eta=E_\epsilon-\epsilon E$ would generate an effective ray, a contradiction.  \end{proof} 

\begin{remark}\label{rem:main}
We will typically apply Proposition \ref {prop:main} and Corollary \ref {cor:main} with $\mu_1=\ldots=\mu_s=1$, $\delta=\sqrt s$ and $\nu_1=\ldots=\nu_s=\frac {m_1}s$.  If
either $s=4, 9$ or $s\ge 10$ and (N) holds, then hypotheses (ii) and (iii) of Corollary \ref {cor:main} hold, and ${\sum_{i=1}^s \nu_i^2}=1$.
Hence, if $C=(d;m_1,\ldots,m_t)$ generates a good  ray in $N_1(X_t)$,  then  $C_\frac s\eta=(d; (\frac {m_1}\eta)^s, m_2,\ldots,m_t)$ generates a noneffective ray  for all rational numbers $\eta\le \sqrt s$, therefore the ray 
$[d; (\frac {m_1}{\sqrt s})^s, m_2,\ldots,m_t]$ is either good or wonderful, in particular it is nef. 

 In this situation, if $C$ is standard and $s\ge 9$, then $C_\delta$ is also standard. The same holds for $s=4$ if $2d\ge 3m_1$. This will be the case for the examples we will provide to prove Theorem \ref {thm:main}, so all of them will be standard. 
\end{remark}

The \emph{base of the induction},  consists in exhibiting SNS multiplicity vectors for $10\le n\le  12$, giving rise to good rays as indicated at the end of \S~ \ref {ssec:more}. They will provide the starting points of the induction for proving Theorem \ref {thm:main} (see \S \ref {ss:proof}). This step is  based on a slight  improvement of the same degeneration technique  used to prove Proposition \ref {prop:main} (see \S \ref {ssec:second}).

\begin{remark} To the best of our knowledge, it is only for a square number of points 
that SNS multiplicity vectors  and good rays
were known so far:  i.e., 
$[d;1^{d^2}]$ is a good ray (see \cite{Nag59}) and 
$(1^{d^2})$ is an SNS multiplicity vector for $d\ge 4$ (see \cite{CM06, Eva07, Roe??}).
\end{remark}

\begin{example}\label{ex:main} Using the goodness of $[d;1^{d^2}]$ and applying Corollary  \ref  {cor:main}, we see that all rays of the form $[dh; h^{d^2-\ell}, 1^{\ell h^2}]$, with $d\ge 4$, $h\ge 1$ and $0\le \ell\le d^ 2$  integers, are good. 
\end{example}

\subsection{The basic degeneration}\label{ssec:1stdeg}
We briefly recall the degeneration we use to prove Theorem \ref {thm:main} (see  \cite {CDMR11, CM98a, CM11}  for details).

Consider $Y \to \mathbb D$ the family
obtained by blowing up  the trivial family $\mathbb D\times {\mathbb{P}}^2\to \mathbb D$ over a disc $\mathbb D$
at a point in the central fiber.
The general fibre $Y_u$ for $u\neq 0$ is a ${\mathbb{P}}^2$,
and the central fibre $Y_0$ is the union of two surfaces
$V \cup Z$, where $V \cong {\mathbb{P}}^2$ is the exceptional divisor and
$Z \cong {\mathbb{F}}_ {1}$ is the original central fibre blown up at a point.
The surfaces $V$ and $Z$ meet transversally along a rational curve $E$
which is the negative section on $Z$ and a line on $V$.

Choose $s$ general points on $V$ and $t-1$ general points on $Z$.
Consider these $n=s+t-1$ points as limits of $n$ 
general points in the general fibre $Y_u$
and blow these points up in the family $Y$, getting a new 
family. We will abuse notation and still denote by $Y$ this new family.
The blow-up creates $n$ exceptional surfaces $R_i$, $1\le i\le n$,
whose intersection with each fiber $Y_u$ is a $(-1)$-curve,
the exceptional curve for the blow-up of that point in the family.
The general fibre $Y_u$ of the new family
is an $X_n$.
The central fibre $Y_{0}$ is the union of $V$ blown-up at 
$s$ general points, and $Z$  blown-up at $t-1$ general points. We will abuse notation
and still denote by  $V$ and $Z$ the blown-up surfaces which are now isomorphic to $X_s, X_t$ respectively. 

Let $\cO_Y(1)$ be the pullback on $Y$ of $\cO_{\bbP^2}(1)$. 
Given a  multiplicity vector $(m_1, \dots,  m_{n})$, a degree $d$ and a \emph{twisting integer} $a$,
consider the line bundle
\[\cL(a)=\cO_Y(1)\otimes \cO_Y(-m_1R_1)\otimes \cdots \otimes \cO_Y(-m_nR_n))\otimes\cO_Y(-aV).\]
Its restriction to $Y_u$ for $u \ne 0$ is 
$(d;m_{1},...,m_{n})$. 
Its restrictions to
$V$ and $Z$ are $\cL_V=(a; m_{1},...,m_{s})$, 
$\cL_Z=(d; a,m_{s+1},...,m_{n})$ respectively. 
Every \emph{limit line bundle}  of $(d; m_{1},...,m_{n})$ on $Y_t$  is the 
restriction to $Y_0=V \cup Z$ of $\cL(a)$ for an integer $a$. 

We will say that a line bundle $\cL(a)$ is \emph{centrally effective}
if its restriction to both $V$ and $Z$ is effective.

\begin{theorem} [The Basic Non-Effectivity Criterion (BNC), (see \cite {CM11})] 
If there is no twisting integer $a$
such that $\cL(a)$ is centrally effective, then $(d;m_{1},...,m_{n})$ is non-effective.
\end{theorem}

\subsection{The proof of the induction step} In this section we use (BNC) to give the:

\begin{proof} [Proof of Proposition \ref {prop:main}]. We need to prove 
that $(xd; x\eta\nu_1,\ldots, x\eta\nu_s, xm_2,\ldots, xm_t)$ is not effective for all positive integers $x$. 

In the  setting of \S \ref {ssec:1stdeg}, fix the multiplicities on $V$ to be 
$x\eta\nu_1,\ldots, x\eta\nu_s$ and on $Z$ to be 
$xm_2,\ldots, xm_t$.  We argue by contradiction, and assume there is a central effective $\cL(a)$. Then $\cL_V=(a; x\eta\nu_1,\ldots, x\eta\nu_s)$
is effective hence $D\cdot \cL_V\ge 0$, i.e.  $a\delta\ge x\eta m_1$,  therefore $a\ge xm_1$. 
Since $\cL_Z=(xd; a, xm_2,\ldots, xm_t)$ is effective, so  is $(xd; xm_1, xm_2,\ldots, xm_t)$ which contradicts 
$C=(d;m_1,\ldots,m_t)$ not being effective.  
\end{proof}

 \subsection{$2$-throws}\label{ssec:second}  To deal with the base of the induction, we need to analyse the \emph{matching} of the sections of the  bundles $\cL_V$ and $\cL_Z$ on the double curve. For this we need a modification 
 of the basic degeneration, based on the concept of a \emph{$2$-throw}, described in \cite{CM11}, which we will briefly recall now.  
 In doing this we will often abuse notation, which we hope will create no problems for the reader.

Consider a degeneration of surfaces over a disc, with central fibre containing two components $X_1$ and $X_2$
 meeting transversally along a double curve $R$.
Let $E$ be  a $(-1)$-curve on $X_1$ that intersects  $R$ transversally at two points.
Blow it up in the threefold total space of the degeneration.
The exceptional divisor  $T\cong {\mathbb{F}}_{1}$
meets $X_1$ along $E$,  which is the negative section of  $T$.
The surface $X_2$ is blown up twice, with two exceptional divisors $G_1$ and $G_2$.

Now blow-up $E$ again, creating a double component $S \cong {\mathbb{P}}^1\times {\mathbb{P}}^1$ of the central fibre
that meets $X_1$ along $E$
and $T$ along its negative section. 
The blow-up affects $X_2$,
by creating two more exceptional divisors $F_1$ and $F_2$
which are $(-1)$ curves on $X_2$, while 
$G_1$ and $G_2$ become $(-2)$-curves.
Blowing $S$ down by the other ruling contracts $E$ on the surface $X_1$.
The curve $R$ becomes nodal, and $T$ changes into a ${\mathbb{P}}^2$.
The surface $X_2$ becomes non-normal, singular along the identified $(-1)$-curves
$F_1, F_2$.

On $X_2$  we introduced two pairs of \emph{infinitely near points}
corresponding to the $(-1)$-curves $F_i$ and $F_{i}+G_{i}$,  which  is also a curve with self-intersection $-1$, 
and we call  $F_i$ and $F_{i}+G_{i}$ a pair of \emph{infinitely near} $(-1)$-curves, with $1\le i\le 2$. 
We denote the assignment of multiplicities  to a pair of infinitely near points as above  by $[a,b]$,
indicating a multiple point $a$ and an infinitely near multiple point $b$,
namely $-a(F_i+G_i)-bF_i$.

\subsection{The base of the induction}\label{ssec:base} 

The above discussion is general. In order to deal with the base of the induction, 
we will now apply it to the degeneration 
$V\cup Z$ described in section \ref{ssec:1stdeg}, with $n=10$ (for the cases $11\le n\le 12$ the basic degeneration, plus some more care on the matching, suffices). The proofs here are quite similar
to the ones in \cite {CDMR11,CM11}, hence we will be brief.  

\subsubsection{The $n=10$ case} 
\begin{proposition} \label {prop:10}
The multiplicity vector $(5,4^9)$ is (SNS). In particular $B_{1,2}=(13; 5,4^9)$ generates a good ray.
\end{proposition}
\begin{proof} It suffices to prove that, for every positive integer $x$, $(13x;5x,(4x)^9)$ 
is non-effective and $(13x+1;5x,(4x)^9)$ is non-special. 

We will show that $(15x;6x,(4x)^3,(5x)^5,4x)$  
((CK)-equivalent to $(13; 5x,(4x)^9)$)
is not effective. We  assume by contradiction that the linear system 
is effective for some $x$. 

Consider first the basic degeneration
with  $s=4$, $t=7$, endowed with the line bundles  $\cL(a)$ as in \S \ref {ssec:1stdeg}.
Then perform the $2$-throw of the $(-1)$-curve  $E=(3;2,1^6)$ on $Z$ (see \S \ref {ssec:second}). 
The normalization of $V$ is a $8$-fold blow up of $\bbP^2$, 
two of the exceptional divisors being identified in $V$.
More precisely, the normalization of  $V$ is the blow-up of the plane at 8 points: 4 of them are in general position,
4 lie on a line, and two of them are infinitely near. 
It is better to look at the surface $Z$ before blowing down $E$. Then $Z\cong X_7$. 
Finally, by executing the  $2$-throw we introduce a plane $T$.

We record that the pencil $P_V=(5; 3,2^3,[1,1]^2)$ on the normalization of $V$ 
and the pencil $P_Z=(3;2,1^5,0)$ on $Z$ are nef.

We abuse notation and still denote by  $\cL(a)$ the pullback  of this bundle to the total space
of the family obtained by the double blow-up of $E$  (see \S \ref {ssec:second}). 
For each triple of
integers $(a,b_1,b_2)$, we can consider the bundle
$\cL(a,b_1,b_2)=\cL(a)\otimes\cO_Y(-b_1T-(b_1+b_2)S)$. We will still denote by
$\cL(a,b_1,b_2)$ the pushout of this bundle to the total space of the $2$-throw family. 
Every limit line bundle of $(15x;6x,(5x)^5,(4x)^4)$
has the form $\cL(a,b_1,b_2)$.

We are interested in those $\cL(a,b_1,b_2)$ which are
centrally effective.  The computations of \cite{CDMR11}, 
specialized to the present case, show that 
$b_1\ge -8x+a$, $b_1+b_2\ge -16x+2a$, and  for every
$b_1, b_2$ satisfying these inequalities,
the restriction of  $\cL(a,b_1,b_2)$ to $Z$ and $V$ are
subsystems of 
\[ {\mathcal{L}}_{Z}=
(63\,x-6\,a; 32\,x-3\,a,(21\,x-2\,a)^5,20\,x-2\,a), \; 
{\mathcal{L}}_{V}=(a; 6x,(4x)^3, [-8\,x+a,-8\,x+a]^2).\]
It suffices to see that there is no value of $a$ which makes both $\cL_Z$ 
and $\cL_V$ effective, and for which there are divisors in these two systems which agree on the double curve.

For $\cL(a,b_1,b_2)$
to be centrally effective one needs
\[ \mathcal{L}_Z\cdot P_Z=20x-2a\ge 0, \;\;   
\mathcal{L}_V\cdot P_V=a-10x\ge 0.\]
This forces $a=10x$ and the restriction of  $\cL(a,b_1,b_2)$
to $Z$ and $V$ are  \emph{equal} to
\[ \cL_Z= (3x; 2x,x^5,0), \;\;
\cL_V=(10x;6x,(4x)^3,[2x,2x]^2),\]
which means $b_1=b_2=-8x+a=2x$, hence the restriction
of the line bundle to $T$ is trivial, and the systems $\cL_Z$ and $\cL_V$ are composed with the
pencils $P_V$ and $P_Z$, thus $\dim(\cL_Z)=x$, $\dim(\cL_V)=2x$.

 Focusing on $\cL_V$, we only need to consider the subspace of sections that match along 
$F_1$ and $F_2$. 
Since the identification $F_1=F_2$ is done via a sufficiently general
projectivity,  this vector space has dimension 1
(see \cite[\S 8] {CM11} for details). 
Then, by transversality on $Z \cap V=R$ (see \cite [\S 3] {CM98a}), and since 
$\cL_{Z\vert R}$ has dimension $x$ and degree $2x$, no section
on $Z$ matches the  one on $V$ to create a section on $X_0$.

Now consider $(13x+1;5x,(4x)^9)$ and its (CK)-equivalent system $(15x+3; 6x+2,(4x)^3,5x,(5x+1)^4,4x)$.
A similar analysis as before,
using $a=10x+1$, leads 
to the following limit
systems on $Z$ and $V$ (trivial on $T$)
 \[ \cL_Z=(3x+12; 2x+7,(x+4)^4,x+3,3), \;\; 
\cL_V=(10x+1; 6x+2,(4x)^3,[2x-1,2x-2]^2).\]

The system $\cL_Z$ is (CK)-equivalent to $(x+5; x+2,2,1^4)$ 
so it is nef, non-empty of the expected 
dimension (see \cite{Har85}). The system $\cL_V$ is also non-empty of the expected dimension:
it consists of  three lines 
plus a residual system (CK)-equivalent to the nef
system $(2x+3; 3,2,1^2,2x-2)$. Thus to compute the
dimension of the limit system as in \cite{CM98a}, it remains
to analyse the restrictions to $R$ (or rather, the \emph{kernel
systems} $\hat \cL_Z$, $\hat \cL_V$ of such restrictions). 
Since both surfaces are  anticanonical, this can be
done quite easily, showing that  they are non-special with $\dim ( \hat \cL_Z)=2x+4$ and
$\dim ( \hat \cL_V)=10x-7$. Thus 
 \cite [3.4, (b)] {CM98a}, applies and non-speciality of 
 $(15x+3; 6x+2,(4x)^3,5x,(5x+1)^4,4x)$ follows.
\end{proof}

\subsubsection{The $n=11$ case} 
\label{sec:11}
\begin{proposition} \label {prop:11}
The multiplicity vector $(3,2^{10})$ is  (SNS). In particular $(7;3,2^{10})$ generates a good ray.
\end{proposition}

\begin{proof}
We prove that $(7x+\delta; (2x)^{10},3x)$
is non-effective for all $x$ and $\delta=0$ and non-special for  $\delta=1$.

Consider the basic degeneration as in \S \ref {ssec:1stdeg}, with $s=4$, $t=8$. 
Then $\cL(a)$ restrict as
\[ \cL_V=(a; (2x)^4), \;\;
\cL_Z=(7x+\delta; a,(2x)^6,3x).\] 

Look at  the case $\delta=0$, where we want to prove non-effectivity. (BNC) does
not suffice for this, so we will compute the dimension of a limit system as in \cite{CM98a}. 
To do this,  pick $a=4x$.
The systems $\cL_V$ and $\cL_Z$ are composed with the pencils
$P_V=(2;1^4)$ and $P_Z=(7;4,(2)^6,3)$ respectively (note that
$(7;4,(2)^6,3)$ is (CK)-equivalent to a pencil of lines), and 
$\dim (\cL_V)=2x, \dim (\cL_Z)=x$. 
The restriction to $R$ has degree $4x\ge 2x+x+1$, so by transversality 
of the restricted systems \cite[\S 3] {CM98a}, 
the limit linear system consists of  the kernel systems
$\hat \cL_V=(4x-1; (2x)^4)$ and $\hat \cL_Z=(7x;4x+1,(2x)^6,3x)$.
These are non-effective, because they meet negatively $P_V$ and $P_Z$
respectively. So $(7x;(2x)^{10},3x)$ is non-effective.

For $\delta=1$ pick again $a=4x$. Then $\cL_V$ is the same,
$\cL_Z=(7x+1;4x,(2x)^6,3x)$ and the kernel systems  are both 
nef,  hence they are non-special by 
\cite{Har85}.  Moreover the restriction of 
$\cL_Z$ to $R$ is the complete series of degree $4x$.
Again by transversality as in \cite[\S 3] {CM98a}, the claim follows.
\end{proof}

\subsubsection{The $n=12$ case} 
\label{sec:12}
\begin{proposition} \label {prop:12}
The multiplicity vector $(2^8,1^{4})$ is (SNS). In particular $(6;2^8,1^{4})$ generates a good ray.
\end{proposition}

\begin{proof}
We prove that $(6x+\delta;x^{4},(2x)^8)$
is empty for all $x$ and $\delta=0$ and non-special for $\delta=1$.

Consider the degeneration of  \S \ref {ssec:1stdeg},
with $s=4$, $t=9$.
Then $\cL(a)$ restrict as  
\[ \cL_V= (a; (x)^4), \;\; \cL_Z=(6x; a,(2x)^8).\]

Let us analyze the case $\delta=0$ for $a=2x$. The system
$\cL_Z$ consists of $2x$ times 
the unique cubic $E$ through the 9 points and
$\cL_V$ is composed with the pencil $P_V=(2,1^4)$,
and its restriction to $R$ is composed with a general pencil of
degree 2. By transversality 
it does not match the divisor cut out by $2xE$ on 
$R$ and the limit system is formed by the
kernel systems. An elementary computation shows that they are not effective.

For $\delta=1$ and $a=2x$, $\cL_V$ is the same, whereas
$\cL_Z=(6x+1; (2x)^9)$ and the kernel 
$\hat \cL_Z=(6x+1;2x+1,(2x)^8)$ are both 
nef, hence they are non-special by 
\cite{Har85}. As before, the restriction of 
$\cL_Z$ to $R$ is the complete series of degree $2x$,
and transversality gives the claim.
\end{proof}

\begin{remark}\label{rem:prob} (i) In Example \ref {ex:main} we saw that  $[dh; h^{d^2-\ell}, 1^{\ell h^2}]$, with $0\le \ell\le h$, and $d\ge 4$, $h\ge 1$, is  a good ray. Proposition \ref  {prop:12} shows that if  $d=3$, $h=2$, $\ell=1$, the ray is still good. By Corollary \ref {cor:main}, this implies that 
 if  $d=3$, $h=2$, $1\le \ell\le 9$, the ray is still good.  With a similar argument, one sees that all cases $d=3$, $h\ge 1$, also give rise to good rays. We leave this to the reader.

(ii) For any integer $d\ge 6$, take positive integers $r, s$ such that $d^2=4s+r$. The ray generated by $(d;2^ s,1^ r)$ on $X_n$ is nef. Indeed, one has $(d-1)(d-2)\ge 2s$, so there exists an irreducible curve $C$ of degree $d$ with exactly $s$ nodes $p_1, \dots,p_s$ (\cite{Sev68}, Anhang F). On the blow--up of the $s$ nodes and $r$ other points $q_1,\ldots, q_r$ of the curve, the proper transform of the curve  is a prime divisor of selfintersection zero, thus nef.

If $d=2k$ is even then $r=4k^2-4s$ is a multiple of
four and $(d;2^s,1^r)=(2k;2^{k^2-\ell},1^{4\ell})$ generates a good ray
with $\ell=k^2-s$ by example \ref {ex:main} (because of (i) we may assume $d>6$). This suggests that that the ray $[d;2^ s,1^ r]$ may always be good. To prove it,  taking into account  Corollary \ref{cor:main}, it would suffice to show that $[2k+1; 2^{k^2+k},1]$ is good for all $k\ge 3$. 
\end{remark}

\subsection{The proof of the ET}\label {ss:proof} For $10\le n\le 12$ the problem is settled by Propositions \ref {prop:10}, \ref {prop:11} and
\ref {prop:12}.  To cover all $n\ge 13$,  we apply Corollary \ref {cor:main} with $s=4$, $D=(2;1^4)$ and $\nu_1=\ldots=\nu_4=\frac {m_1}4$ (see Remark \ref {rem:main}). For instance one finds the good rays

  \[
 \begin{aligned}
& [13\cdot 2^h; 5^4, (5\cdot 2)^3, \ldots, (5\cdot 2^{h-1})^3, (2^{h+2})^9]&\text {if }\; n=10+3h,&\; \text {for }  h\ge 1,\cr
 & [7\cdot 2^h; 3^4, (3\cdot 2)^3, \ldots, (3\cdot 2^{h-1})^3, (2^{h+1})^{10}]& \text {if }\; n=11+3h,&\; \text {for }  h\ge 1,\cr
   &[6\cdot 2^{h-1}; (2)^3, \ldots, (2^{h-1})^3, (2^{h})^7, 1^8]&\text {if }\; n=12+3h,&\; \text {for } h\ge 1.\cr
  \end{aligned}
  \]

\renewcommand{\thesection}{\arabic{section}}
\renewcommand{\theequation}{\thesection.\arabic{equation}}
\setcounter{equation}{0}

\renewcommand{\thetheorem}{\thesection.\arabic{theorem}}
\setcounter{theorem}{0}

\section{An application}\label{sec:appl}

\begin{proposition} \label{prop:appl} If (SN) holds for $n=10$ and (N) holds for all $n\le 89$ then (N) holds for all $n\ge 90$.
\end{proposition}

The proof is based on the following:

\begin{lemma}\label{lem:appl}  Assume (SN) holds for $n=10$.
Let $n=s_1+\cdots +s_{10}$, where $s_1,\ldots, s_{10}$ are positive integers such that the Nagata ray $\nu_{s_i}$ is nef for $1\le i\le 10$ and
\begin{equation}\label{eq:appl}  
3\sqrt n\le \sum_{i=1}^{10} \sqrt {s_i}.
\end{equation}
Then $\nu_n$ is nef.
\end{lemma}

\begin{proof}  Consider the ray $[ \sqrt n ,  \sqrt {s_1}, \ldots,   \sqrt {s_{10}}]$
which, by the hypotheses, is in  $\partial\cQ_{10}^\succcurlyeq$. We can approximate it  by good  rays  (see Proposition \ref {prop:SN10}).  By Corollary \ref {cor:main}, we see that $\nu_n$ is the limit of nef rays, hence it is nef. \end{proof}

\begin{proof} [Proof of Proposition \ref {prop:appl}.] We argue by induction. Let $n\ge 90$, and write $n=9h+k$, with $9\le k\le 17$ and $h\ge 9$. By induction both $\nu_h$ and $\nu_k$ are nef.  Moreover \eqref {eq:appl} is in this case $3\sqrt {9h+k}\le 9\sqrt h +\sqrt k$,  which reads $h\ge \frac {16}{81}k$, which is verified because $k\le 17$ and $h\ge 9$. 
Then $\nu_n$ is nef by Lemma \ref {lem:appl}. \end{proof}

\begin{remark}\label{rem:appl} Lemma \ref {lem:appl} is reminiscent of the results in \cite {BZ} and \cite{RR09}. 

The hypotheses in Proposition \ref {prop:appl} can be weakened. For instance, Lemma \ref {lem:appl} implies that, if (SN) holds for $n=10$, then $\nu_{13}$ is nef. Actually, it suffices to know that  $[ \sqrt {13} ;  2, 1^9]$ is nef. As in Example \ref {ex:seq}, 
we may take a sequence $\{(q_n,p_n)\}_{n\in \bbN}$ such that  $\frac {p_n+2q_n}{p_n}$ are the convergents of the periodic continued fraction expansion of $\sqrt {13}=[3; \overline{1^3,6}]$, so that 
\[ p_1=2, \; p_2=3, \; p_3=5, \; p_4=20 ,\ldots \;\; q_1=1, \; q_2=2, \; q_3=3, \; q_4=13  \ldots. \]
The sequence of rays $\{[B_{q_n,p_n}]\}_{n\in \bbN}$ converges to $\nu_{13}$. If we knew that the rays of the sequence are good, this would imply (N) for $n=13$.  Note that $B_{q_1,p_1}=(13;5, 4^9)$ generates  a good ray by Proposition \ref {prop:10}. 

Similarly, if (SN) holds for $n=10$, then $\nu_{n}$ is nef for $n=10h^2$, etc.
We do not dwell on these improvements here.
\end{remark}

\end{document}